\documentclass[a4paper,12pt]{amsart}
\usepackage[utf8]{inputenc}
\usepackage{mathtools}

\makeatletter
\@namedef{subjclassname@2020}{\textup{2020} Mathematics Subject Classification}
\makeatother

\allowdisplaybreaks[1]

\usepackage{amsmath,amsthm,amssymb}
\usepackage{hyperref}
\usepackage{amsfonts}
\usepackage{amsmath}
\usepackage{mathrsfs}
\usepackage{tabto}
\usepackage{changepage} %For adjustwidth
\usepackage{fullpage} %To decrease margins
\usepackage{graphicx}
\usepackage{enumitem} %For itemising

\newcommand{\C}{\mathbb{C}}

\newcommand{\N}{\mathbb{N}}

\numberwithin{equation}{section}
\newtheorem{theorem}{Theorem}[section]
\newtheorem{lemma}[theorem]{Lemma}

\newtheorem{remark}[theorem]{Remark}
\newtheorem{example}[theorem]{Example}
\newtheorem{definition}[theorem]{Definition}
\thanks{The research work of the first author is supported by Institute Research Fellowship of BIT Mesra APO/2024-25/49 and the research work of the second author is supported by ANRF(SERB) research grant TAR/2023/000197 }

\begin{document}
	%On Postsingular dynamics of commuting transcendental entire function.
	\title [Postsingular sets of semigroup]{Results on the postsingular set of compositions of transcendental entire functions}
%	On the Dynamics of Commuting Transcendental Entire Functions: Postsingular Behavior and Eremenko’s Conjecture\\
%	On the Dynamics of Commuting Transcendental Entire Functions: Equality of Dynamical Sets, Postsingular Behavior and Eremenko’s Conjecture
	
	\author[M. Kumari]{Manisha Kumari}
	\address{Department of Mathematics, Birla Institute of Technology Mesra
		Ranchi--835 215, India}
	
	\email{phdam10052.24@bitmesra.ac.in}
	
	\author[D. Kumar]{Dinesh Kumar}
	\address{Department of Mathematics, Birla Institute of Technology Mesra
		Ranchi--835 215, India}
	
	\email{dineshkumar@bitmesra.ac.in }
	
	%\begin{frontmatter}
	%\title{On dynamics of the bungee set and the  filled julia set of a  transcendental semigroup}
	%\author[1]{Manisha Kumari }\ead{phdam10052.25@bitmesra.ac.in}
	%\author[1]{Dinesh Kumar\corref{cor1}}\ead{dineshkumar@bitmesra.ac.in}
	%\address[1]{Department of Mathematics, Birla Institute of  Technology  Mesra, Ranchi-835215,  India}
	%% \journal{Journal of Computational Science }
	%\end{frontmatter}

	\keywords{ Escaping set, Fatou set, normal family, postsingularly bounded, postsingularly finite, hyperbolic set.}
	
	\begin{abstract}
	
	In this paper, we study the dynamics of commuting transcendental entire functions $f$ and $g$, where $g=af^p+b$ with $a,b\in\C$, $p\in\N$, and $a\neq 0,1$. We examine how singular values and postsingular sets behave under composition. Within this framework, we show that if one of the functions is postsingularly finite (respectively, postsingularly bounded, hyperbolic), then the other function also has this property, and so do their compositions. As an application, we derive several results concerning transcendental semigroups, including situations in which the Eremenko’s conjecture is satisfied.
%		\keywords{ Escaping set, Fatou set, normal family, postsingularly bounded, postsingularly finite, hyperbolic set.}
%		In this paper, we study the dynamics of commuting transcendental entire functions $f$ and $g$ where $g$ has the form $af^p+b,\,  \, a,b \in \C,\, p\in \N$ and $a\neq 0,1$. We have shown that their escaping set, filled Julia set, and the bungee set are identical. As a special case we obtained that the Julia sets of $f$ and $g$ are also identical. In fact, our result generalizes the results of K.K. Poon and Yang, from 1998.
%		%		
%		%	We study the dynamics of commuting transcendental entire functions. Under certain specific hypothesis, we show that their escaping set, filled Julia set, and the bungee set are identical.
%		We also examine how the singular values and postsingular sets behave under the composition of functions. Within the same framework, we prove that if one of the functions is postsingularly bounded (respectively postsingularly finite), then the other corresponding function also exhibits analogous dynamical characteristics. Similar dynamical characteristics are also exhibited by their compositions. As an application, we derive several results on transcendental semigroups, including particular instances in which Eremenko’s conjecture is satisfied.
	\end{abstract}
	
	%	\keywords{bungee set, filled Julia set, escaping set, completely invariant, asymptotic value, transcendental semigroup}

	\subjclass[2020]{37F10, 30D05}
	
	\maketitle
	\section{Introduction}

	Let $f:\C\to\mathbb{C}$ be a transcendental entire function and let
$f^{n}$ denote the $n$-th iterates of $f$ for all $n\in\mathbb{N}$. Classically, by dividing the complex plane into two basic sets the Julia set and the Fatou set, the dynamical behavior of $f$ has been understood \cite{baker1984wandering}. The Fatou set of $f$ is represented by $F(f)$ and is defined as follows $F(f)=\{z\in\mathbb{C} :\{f^{n}(z)\}_{n\in\mathbb{N}}\, \text{is normal in some neighborhood of} \, z\}$ and the Julia set denoted by $J(f)$ is its complement, that is, $J(f)=\mathbb{C}\setminus F(f)$. In contrast to the Julia set, points in the Fatou set show steady dynamical behavior. The Julia set exhibits sensitive dependence on initial conditions and the dynamical behavior on the Julia set is highly chaotic. The fundamental characteristics and structural elements of these set are well recorded in the literature \cite{bergweiler1993iteration}. In addition to the complex plane's traditional breakdown into Fatou and Julia sets,  points can also be further categorized based on their orbits under iteration. $I(f)$ represents the escaping set of $f$ and is defined by $I(f)=\{z\in\mathbb{C}: f^{n}(z)\to\infty \text{ as } n\to\infty\}$.
It was introduced for the first time by Eremenko \cite{eremenko1989iteration}, and is made up of points whose forward orbit escapes to infinity. Additionally, he demonstrated that every connected component of $\overline{I(f)}$ is unbounded and that $I(f)$ is non-empty. $K(f)$ is the filled Julia set of $f$ and is defined as $K(f)=\{z\in\mathbb{C}:\{f^{n}(z)\}_{n\in\mathbb{N}} \text{ is bounded}\}.$
In other words, it involves points with bounded orbits. Furthermore, the bungee set of $f$ is represented by $BU(f)$ and is defined to be the set of points whose orbit exhibits a mixed dynamical behavior, that is, the orbit admits at least two subsequences, one of which tends to infinity while the other stays bounded.
	
	Two functions $f$ and $g$ are said to be commuting (permutable) if
	$f \circ g = g \circ f$. In \cite{kumar2013dynamics}, the second author examined how the Fatou set and singular	values of transcendental entire functions $f$, $g$ and their composition $f \circ g$ are related. They obtained a variety of conditions ensuring that the Fatou set of $f$ and $f \circ g$ coincide and further analyzed the relationship between the singular values of $f$ and $g$ and those arising from their composition. Eremenko's conjecture \cite{eremenko1989iteration} that each component of $I(f)$ is unbounded has served as a major inspiration for most of the work on $I(f)$. It saw significant advancement in \cite{rippon2005questions}, where it was demonstrated that $I(f)$ always has at least one unbounded component. 
Singular values mostly govern the dynamics of an entire function. We now look at the key definitions of singular values and postsingular sets, which are essential to understand the dynamics of transcendental entire functions. The set of singular values of $f$, represented by $\mathrm{Sing}(f^{-1})$, consist of all asymptotic and critical values of $f$ along with their finite limit points. The postsingular set of $f$ is denoted as $P(f)$ and is defined as the closure of the forward orbit of the singular values of $f$ under iteration, that is,\[P(f)=\overline{\bigcup_{n \geq 0} f^n(\mathrm{Sing}(f^{-1}))}.\] If $P(f)$ is finite (respectively, bounded) then the entire function $f$ is said to be postsingularly finite (respectively, postsingularly bounded).
%The function $f$ is called postsingularly finite if $P(f)$ is finite
 The entire function $f$ is called hyperbolic if its postsingular set $P(f)$ is a compact subset of the Fatou set $F(f)$. Recall that a transcendental semigroup $H$ is a semigroup generated by a family of transcendental entire functions $\{h_1,h_2,\dots \}$, where the semigroup operation is the composition of functions. Semigroup $H$ is said to be finitely generated if there are only finitely many generators. For a transcendental semigroup $H$, the Fatou set $F(H)$ is defined as the maximal open subset of $\C$ on which the family of functions in $H$ is normal. The Julia set $J(H)$ is then defined as the complement of $F(H)$. The escaping set of the transcendental semigroup $H$, denoted by $I(H)$, is given by
 $I(H) = \{ z \in \C \mid \text{for every sequence in } H, \text{ there exists a subsequence that diverges to infinity at } z \}$, \cite{kumar2016dynamics}.\\
	% We have shown that Eremenko's conjecture is true under certain condition on the transcendental semigroup.\\
%	Motivated by these developments, we demonstrate that if $f$ is postsingularly bounded (respectively postsingularly finite, hyperbolic), then so are $g$ and $f \circ g,$ where $g$ is given by $af^p+b, \, a, b \in \C, p \in \N$ and $a\neq 0,1$. In connection with Eremenko's conjecture, we additionally provide a topological analysis of the class of functions discussed above.\\
%	In this paper, we have considered two commuting transcendental entire function $f$ and $g$ where $g$ is given by $af^p+b, \, a, b \in \C, p \in \N$ and $a\neq 0,1$. Using the concept of postsingular set. We have provided the condition under which $g$ and $f\circ g$ are postsingularly bounded (respectively postsingularly finite, hyperbolic) whenever $f$ is postsingularly bounded (respectively postsingularly finite, hyperbolic). In addition we discussed some topological questions related to the semigroup $H$ generated by $f$ and $g$ are postsingularly bounded (respectively postsingularly finite, hyperbolic). Moreover, we have shown that Eremenko's conjecture holds in this situations.\\
	In this paper, we have explored a pair of commuting transcendental entire functions $f$ and $g$, where $g$ has the form $g = af^p + b$ with $a, b \in \C$, $p \in \N$, and $a \neq 0,1$. In addition $|a|=1$. We have established conditions under which $g$ and $f \circ g$ are postsingularly finite (respectively, postsingularly bounded, hyperbolic) whenever $f$ is so. Furthermore, we have  addressed several topological properties of the finitely generated abelian transcendental semigroup $H=[f,g]$, where the generators are postsingularly finite (respectively, postsingularly bounded, hyperbolic). Finally, we have shown that Eremenko’s conjecture holds in this setting.\\
		\textbf{Structure of the paper :} We provide some known definitions and basic facts in section \refeq{sec 2}, which is relevant to our study, including complete invariance and singular values of transcendental semigroup. In section \refeq{sec4}, we have established the equality of the Fatou set and the escaping set of the semigroup $H=[f,g]$ with the Fatou set and the escaping set of its generators. We have also examined that if one of the generator is postsingularly finite (respectively, postsingularly bounded, hyperbolic) then so is the other. Finally, in the last section \refeq{sec5}, we have considered some topological questions about the escaping set of such abelian semigroup $H$ generated by $f$ and $g$. Moreover, using the notion of the postsingular set of a transcendental semigroup we have obtained a partial solution to the Eremenko's conjecture. 
	\section{Preliminaries} \label{sec 2}
	In this section, we recall some important notions of transcendental semigroup theory. A transcendental semigroup $H$ is represented by $H=[h_1, h_2, \dots ]$ and is a semigroup generated by a set of transcendental entire functions $\{h_1,h_2,\dots\}$. The semigroup is closed under functional composition and each element $h \in H$ is a transcendental entire function derived from the finite composition of its generators. 
	%\\
	%A family of transcendental entire functions $\{h_1, h_2,...\}$ are able to generate a transcendental semigroup $H$, where functional composition works as the semigroup operation. The semigroup is represented by $H = [h_1, h_2,...,]$.$ G$ is therefore closed under composition and any $h\in H$ is a transcendental entire function.
	%The iteration theory of entire functions is introduced in \cite{bergweiler1993iteration}.
	Hinkkanen and Martin \cite{hinkkanen1996dynamics}, provided major contribution in this direction. If the generators of $H$ commute with each other, then $H$ is called an abelian
	transcendental semigroup.\\
	The following are some established definitions in transcendental semigroup theory.

	\begin{definition}[Forward and backward invariance]
		Let $H$ be a transcendental semigroup. A set $W \subset \mathbb{C}$ is
		said to be \emph{forward invariant} under $H$ if
		\[
		h(W) \subset W \quad \text{for all } h \in H.
		\]
		The set $W$ is  \emph{backward invariant} under $H$ if
		\[
		h^{-1}(W)=\{ w \in \mathbb{C} : h(w) \in W \} \subset W
		\quad \text{for all } h \in H.
		\]
		If $W$ is both forward and backward invariant under $H$, then $W$ is
		said to be \emph{completely invariant} under $H$ \cite{kumar2016dynamics}.
	\end{definition}
	
	%Recall that a point $w \in \mathbb{C}$ is called a \emph{critical value}
	%of a transcendental entire function $f$ if there exists a point
	%$w_0 \in \mathbb{C}$ such that $f(w_0)=w$ and $f'(w_0)=0$. In this case,
	%$w_0$ is referred to as a \emph{critical point} of $f$. Thus, the image
	%of any critical point under $f$ is a critical value of $f$ \cite{}.
	If there is a point $w_0 \in \mathbb{C}$ such that $f(w_0)=w$ and $f'(w_0)=0$, then $w \in \mathbb{C}$ is referred to as a \emph{critical value} of $f$ and $w_0$ is called  \emph{critical point} of $f$. Therefore, a critical value of $f$ is the image of a critical point under $f$.
	
	Moreover, if there is a curve $\Gamma$ tending to infinity such that $f(z) \to \zeta$ as $z \to \infty$ along $\Gamma$, then $\zeta \in \mathbb{C}$ is considered as an \emph{asymptotic value} of $f$. The corresponding concepts for transcendental semigroups were presented in \cite{kumar2015dynamics}.
	
	\begin{definition}
		A point $z_0 \in \mathbb{C}$ is called a \emph{critical point} of a
		transcendental semigroup $H$ if it is a critical point of at least one
		function $h \in H$. A point $w \in \mathbb{C}$ is called a
		\emph{critical value} of $H$ if it is a critical value of some
		$h \in H$.
		% \cite{kumar2015dynamics}
	\end{definition}
	
	\begin{definition}
		A point $w \in \mathbb{C}$ is called an \emph{asymptotic value} of a
		transcendental semigroup $H$ if it is an asymptotic value of at least one
		function $h \in H$.
		%\cite{kumar2015dynamics}.
	\end{definition}
There are several significant categories of transcendental semigroups for which $I(H)$ is non-empty. We now present an example of a transcendental semigroup $H$ whose escaping set $I(H)$ is non-empty.     
	\begin{example}\label{ex 1}
		%
		%Let $f=\lambda Sinz, \lambda \in \mathbb{C}\setminus {0}$ and $ g = af + b$, where $a,b \in \mathbb{C}, \, |a|=1$ and $a\neq 1$. Let $H=[f, g]$ be a transcendental semigroup.
		% we obtain for any $n \in \mathbb{N}$,
		%\[
		%g^n(z) = a^n f^n(z) + b \sum_{k=0}^{n-1} a^k.
		%\]
		%so, $I(f)=I(g)$.
		%One can seen that for $l,m,n,p \in \mathbb{N}$, $f^l\circ g^m= f^l(a^m f^m +b \sum_{k=0}^{m-1} a^k)$ and $g^n\circ f^p= a^n f^{n+p} +b \sum_{k=0}^{n-1} a^k $. As a result any $h\in H $ is either $h=f^s$ for some $s\in \mathbb{N}$ or $h=a^nf^n +b(1+a+a^2+\dots a^{n-1})= g^n$ for some $n\in \mathbb{N}$. In both scenario $I(h)=I(f)$. Hence we conclude that $I(H)=I(f)\neq \emptyset$. \\ 
		Let $f=\lambda \, \sin z$ with $\lambda \in \mathbb{R}\setminus \{0\}$, and $g(z) = af^p(z) +b$ where, $ a,b \in \mathbb{C},\,  p\in \mathbb{N} $ and $a \neq 0,1$. Consider the abelian transcendental semigroup $H=[f,g]$. It can be seen that for any $n \in \mathbb{N}$, 
		\[
		g^n(z) = a^n f^{np}(z) + b \sum_{k=0}^{n-1} a^k.
		\]
	As $|a|=1$, $a\neq 1$, and the term $b \sum_{k=0}^{n-1} a^k$ is bounded, the growth of $g^n$ and $f^{np}$ are equivalent. Therefore, $I(f)=I(g)$. Now, for $l,m,n,q \in \mathbb{N}$, we obtain 
		%$f^l\circ g^m = f^l\big(a^m f^{mp} + b \sum_{j=0}^{m-1} a^j\big) \implies a^m f^{mp+l} + b \sum_{j=0}^{m-1} a^j$
		$f^l\circ g^m = a^m f^{mp+l} + b \sum_{j=0}^{m-1} a^j$
		and
		$
		g^n\circ f^q = a^n f^{np+q} + b \sum_{k=0}^{n-1} a^k.
		$
		%where $ \, c_n=b \sum_{k=0}^{n-1} a^k$.
		%Consequently, every $h \in H$ is of the form $h(z)= a^n f^{np+m}(z) + c_n$, where $ \, c_n=b \sum_{k=0}^{n-1} a^k$.
		%
		%We have $I(f)=I(g)$, since $g^n$ and $f^n$ have the same growth. For $n,r \in \mathbb{N}$,
		%\[
		%g^n\circ f^r = a^n f^{np+r} + c_n,
		%\]
		%where $c_n = b \sum_{k=0}^{n-1} a^k$.
		Thus, every $h \in H$ has the form $h = a^n f^{np+r} + c_n$, where $c_n = b \sum_{k=0}^{n-1} a^k$ and for some $r\in \N$. We have either
		$h = f^s$ for some $s \in \mathbb{N}$, or
		$h = a^n f^{np}+ b \sum_{k=0}^{n-1} a^k = g^n$
		for some $n \in \mathbb{N}$. This implies that $I(f)= I(h)$. Therefore, $I(H)=I(f)\neq \emptyset$.
	\end{example}
	\section{Postsingular behavior under commuting entire functions}\label{sec4}
	
	% we demonstrate that for a particular class of functions, if one of them is postsingularly bounded (respectively postsingularly finite and hyperbolic), then the other related functions and their compositions are also postsingularly bounded (respectively postsingularly finite and hyperbolic). Also,
 In the present section, we show that the escaping set and the Fatou set of the semigroup coincide with those of its generators. The transcendental semigroup $H$ is generated by $f$ and $g$, where $g$ is given by $af^p(z) + b$ with $a,b \in \mathbb{C}$, $p \in \mathbb{N}$, and $a \neq 0,1$. We first establish a lemma which gives relation between $f, g$ and the affine map $P(z)=az+b, \, 0\neq a,b \in \C$.
	\begin{lemma}\label{l2}
		Suppose $f$ and $g$ are two commuting transcendental entire functions where $g=af^p+b,\, a,b\in \mathbb{C}$, $p\in \mathbb{N}$, and $a \neq 0,1$. Then, $f$ commutes with $az+b,$ i.e.
		\[
		f(az+b)=af(z)+b.
		\]
	\end{lemma}
	
	\begin{proof}
		We have $f\circ g(w)=g\circ f(w)$.
		As $g(w)=af^p(w)+b$, we obtain
		\begin{align*}
			f\circ g(w)
			&= f(af^p(w)+b).
		\end{align*}
		On the other hand, we have $g\circ f(w)=af^p(f(w))+b$. Substituting $f^p(w)=z$, this leads to the following functional equation $f(az+b)=af(z)+b$.
	\end{proof} 

	\begin{theorem}\label {Th 1}
		Suppose $f$ and $g$ are commuting transcendental entire functions, where $g(z) = af^p(z) + b, \, a,b \in \mathbb{C}, \, p\in \mathbb{N}$ and $a \neq 0,1$.  
		Consider the abelian transcendental semigroup $H = [f,g]$. Then, for every $h \in H$, we have
		\[
		F(H) = F(h) \quad \text{and} \quad I(H) = I(h).
		\]
	\end{theorem}
	\begin{proof}
		%As given $g=af+b$ then $g^n(z) = a^n f^n(z) + b \sum_{k=0}^{n-1} a^k$  for $n\in \mathbb{N}$ so, $F(f)=F(g)$ and $I(f)=I(g)$. From example \ref{ex 1} we have $g^n\circ f^p = a^n f^{n+p} + b \sum_{k=0}^{n-1} a^k$. Therefore any $h\in H$ is of the form $h= g^n\circ f^p = a^n f^{n+p} + b \sum_{k=0}^{n-1} a^k$ for some $n,p \in \mathbb{N}$. Thus $I(h)=I(f^{n+p})=I(f)$ for all $h\in H$ on similar line we have $F(h)=F(f^{n+p})=F(f)$ for all $h\in H$. Hence, $I(H)=I(f)$ and $F(H)=F(f)$ for all $h\in H$. 
		
		% $g = af^p + b$ gives
		%\[
		%g^n(z) = a^n f^{np}(z) + c_n \quad \text{for } n \in \mathbb{N},
		%\]
		% where, $c_n= b \sum_{k=0}^{n-1}a^k$. The iterates of $f^{np}$ and $g^n$ are comparable, since $|a|=1$ then the sequence $c_n$ is bounded.
		%It follows that $F(f) = F(g)$ and $I(f) = I(g)$. From Example \ref{ex 1} we know that $g^n \circ f^m(z) = a^n f^{np+m}(z) + c_n$. Therefore, every $h \in H$ can be written as $h = g^n \circ f^m(z) = a^n f^{np+m}(z) + c_n$,
		%for some $n, p \in \mathbb{N}$. Consequently,
		%\[
		%I(h) = I(f^{\,np+m}) = I(f) \quad \text{for all } h \in H,
		%\]
		%and similarly,
		%\[
		%F(h) = F(f^{\,np+m}) = F(f) \quad \text{for all } h \in H.
		%\]
		%Thus we conclude that $F(H) = F(h)$ and $I(H) = I(h)$.
		
	By repeatedly computing the iterates of $g = af^p + b$, we obtain by induction
		\[g^n(z) = a^n f^{np}(z) + c_n \quad \text{for } n \in \mathbb{N}, \] where $c_n= b \sum_{k=0}^{n-1}a^k$. As $|a|=1$ and $a\neq 1$ the term $c_n$ is bounded. Hence, the iterates $f^{np}$ and $g^n$ shows similar behavior which implies that $F(f) = F(g)$ and $I(f) = I(g)$. It can be easily seen that for some $n, m \in \mathbb{N}$, every $h \in H$ may be expressed as $h = g^n \circ f^m(z) = a^n f^{np+m}(z) + c_n$. Thus, \[ I(h) = I(f^{\,np+m}) = I(f) \quad \text{for all } h \in H, \] and likewise, \[ F(h) = F(f^{\,np+m}) = F(f) \quad \text{for all } h \in H. \]
		Consequently, we deduce that $I(H) = I(h)$ and $F(H) = F(h)$.
	\end{proof}
%	\begin{remark}
%	On similar lines, it can be easily seen that for an abelian transcendental semigroup $H=[f,g]$ where $g = af^p + b, \, a,b\in \C, p\in \N$ and $a\neq 0,1$. Then for every $h\in H$, we have $BU(H)=BU(h)$ and $K(H)=K(h)$.
%	\end{remark}
The conclusion of the above theorem is still true for non-abelian transcendental semigroup as shown in the given remark.
	\begin{remark}
		Suppose $f$ is periodic of period $c$ and $g=f^q+c$ for some $ q\in \mathbb{N}$. Then $f \circ g \neq g \circ f $. If we consider the semigroup $H=[f,g]$ then, the conclusion of the above theorem holds as shown in \cite{kumar2016dynamics}.
	\end{remark}

We now recall the Eremenko--Lyubich class
	\[
	\mathcal{B}
	=
	\left\{
	f : \mathbb{C} \to \mathbb{C} \ \text{transcendental entire}
	\; : \;
	\mathrm{Sing}(f^{-1})  \text{is bounded}
	\right\},
	\]
	where $\mathrm{Sing}(f^{-1})$ represents the collection of all critical values and asymptotic  values of $f$ along with their associated finite accumulation points. Dynamics is to a large extent controlled by the presence of singular values. Any function $f \in \mathcal{B}$ is said to be of
	bounded type . It is known that class $\mathcal{B}$ is closed
	under composition that is, if $f$ and $g$ are of bounded type, then their composition $f \circ g$ is also of bounded type  \cite{bergweiler1998dynamics}.

Recall that an entire function $f$ is called postsingularly finite if every singular value of $f$ has a finite forward orbit that is, each singular orbit is pre-periodic. Furthermore, we show that if one of the commuting entire function is postsingularly finite, then so is the other and their compositions.
	%	 Next we extend the notion of postsingular finiteness to the setting of semigroups.  
	%	\begin{definition} A transcendental semigroup $H$ is said to be postsingularly finite if every function $h\in H$ is postsingularly finite.  
		%	\end{definition}
	%	From \cite[ Theorem 1.1]{rempequestion}, it follows that if $f\in \mathcal{B}$ is postsingularly finite, then every component of $I(f)$ is unbounded. In what follows, we demonstrate that this statement can be generalized to a certain class of postsingularly finite semigroups.
	%\begin{theorem}
	%Suppose $f \in \mathcal{B}$ and $f$ is postsingularly bounded. Let $g = af + b,\quad a,b \in \mathbb{C}, |a|=1$ and let $H = [f,g]$. Then $H$ is postsingularly bounded and all components of $I(H)$ are unbounded.
	%\end{theorem}
	%\begin{proof}
	
	%\end{proof}
	\begin{theorem} \label{th3}
		Suppose $f$ and $g$ are commuting transcendental entire functions, $f\in \mathcal{B}$,
		$g(z) = a f^p(z) + b, \quad a,b \in \mathbb{C}, \, p\in \mathbb{N},  \, and \,\, a \neq 0,1.$
		If $f$ is postsingularly finite, then so are $g$ and $f \circ g$.
	\end{theorem}

	\begin{proof}
		Consider the polynomial $P(z)=az+b$. Then $P$ commutes with $f$. Also, we have $g=P\circ f^p$ which implies $g^n=P^n \circ f^{np}$. Moreover, $\mathrm{Sing}(g)^{-1}= a \, \mathrm{Sing}(f^p)^{-1}+b$. Now,
		\begin{align*}
			P(g)
			&=\overline{\bigcup_{n \geq 0}g^n \, \mathrm{Sing}(g^{-1})}\\
			&= \overline{\bigcup_{n \geq 0}P^n\circ f^{np}(a \, \mathrm{Sing}(f^p)^{-1}+b)}\\
			&= \overline{\bigcup_{n \geq 0} f^{np}\circ P^n(a \, \mathrm{Sing}(f^p)^{-1}+b)}\\
			%	&= \overline{\bigcup_{n \geq 0}f^{np} (a^n\, (\mathrm{Sing}(f^p)^{-1})+b)}\\
			&= \overline{\bigcup_{n \geq 0} f^{np} \left(a^{n+1} \mathrm{Sing}(f^p)^{-1} + a^nb +  b \sum_{k=0}^{n-1} a^k\right)}\\ 
			&= \overline{\bigcup_{n \geq 0} a^{n+1} f^{np} \mathrm{Sing}(f^p)^{-1} + a^nb +  b \sum_{k=0}^{n-1} a^k}\\ 
			&=  \overline{\bigcup_{n \geq 0} a^{n+1}f^{np} (\mathrm{Sing}(f^p)^{-1})}+ a^nb + b \sum_{k=0}^{n-1} a^k\\
			&= a^r P(f^p)+ a^nb +  b \sum_{k=0}^{n-1} a^k  , \, \, \, \, 0<r< p,\hspace{.2cm} \text {where a is some $p^{th}$ root of unity.}
		\end{align*} 
		%	 \begin{align*}
			%	 	P(g)
			%	 	&=\overline{\bigcup_{n \geq 0}g^n \, \mathrm{Sing}(g^{-1})}\\
			%	 	&= \overline{\bigcup_{n \geq 0}P^n\circ f^{np}(a \, \mathrm{Sing}(f^p)^{-1}+b)}\\
			%	 	&= \overline{\bigcup_{n \geq 0} f^{np}\circ P^n(a \, \mathrm{Sing}(f^p)^{-1}+b)}\\
			%	 	%	&= \overline{\bigcup_{n \geq 0}f^{np} (a^n\, (\mathrm{Sing}(f^p)^{-1})+b)}\\
			%	 	&= \overline{\bigcup_{n \geq 0} a^n\, f^{np} (a \, \mathrm{Sing}(f^p)^{-1}+b)+ b \sum_{k=0}^{n-1} a^k}\\
			%	 	&=  \overline{\bigcup_{n \geq 0} a^{n+1} f^{np} (\mathrm{Sing}(f^p)^{-1})}+ a^nb + b \sum_{k=0}^{n-1} a^k\\
			%	 	&= a^r P(f^p) + a^nb+ b \sum_{k=0}^{n-1} a^k, \, \, \, \text{ for all} \, 0<r< p,\hspace{.2cm} \text {where a is some $p^{th}$ root of unity}
			%	 \end{align*}
		As $f$ is postsingularly finite and the term $ b \sum_{k=0}^{n-1}a^k$ is bounded since, $|a|=1$ and $a\neq 1$, it follows that $P(g)$ is postsingularly finite. On similar lines, it can be seen that $f \circ g$ is also postsingularly finite. Hence, the result.
	\end{proof}
	Recall that an entire function $f$ is called \emph{postsingularly bounded} if its
	postsingular set $P(f)$ is bounded \cite{kumar2014escaping}.
% Motivated by this notion, we extend
%	the definition to transcendental semigroups.
%	
%	\begin{definition}
%		A transcendental semigroup $H$ is said to be \emph{postsingularly
%			bounded} if every element $h \in H$ is postsingularly bounded.
%	\end{definition}
%	
%	It is shown in \cite{rempequestion} that if $f$ is an entire function of bounded
%	type whose singular values have bounded forward orbits (equivalently,
%	$f$ is postsingularly bounded), then each connected component of the
%	escaping set $I(f)$ is unbounded. This result gives a partial resolution
%	of a conjecture posed by Eremenko \cite{eremenko1989iteration}. In the present work, we
%	demonstrate that an analogous conclusion holds for a certain class of
%	postsingularly bounded transcendental semigroups.
	 We now recall the following result which states that postsingularly bounded entire functions are closed under self compositions.
	
	\begin{lemma} \label{l 1}
		Let $f \in \mathcal{B}$ be postsingularly bounded. Then the iterate
		$f^{k}$ is postsingularly bounded for every $k \in \mathbb{N}$ \cite[Lemma 3.7]{kumar2016dynamics}.
	\end{lemma}
	%Before proving the next result, we first state the following lemma, which clarifies the connection between $f$, $g$, and the affine function $az+b$.
	%\begin{lemma}\label{l2}
	%Suppose $f$ and $g$ are two commuting transcendental entire functions where $g=af^p+b,\, a,b\in \mathbb{C}$, $p\in \mathbb{N}$, and $|a|=1$. Then, 
	%\[
	%f(az+b)=af(z)+b.
	%\]
	%\end{lemma}
	%
	%\begin{proof}
	%We have $f\circ g(w)=g\circ f(w)$.
	%As $g(w)=af^p(w)+b$, we obtain
	%\begin{align*}
	%f\circ g(w)
	%&= f(af^p(w)+b).
	%\end{align*}
	%On the other hand, we have $g\circ f(w)=af^p(f(w))+b$. Setting $f^p(w)=z$, this leads to the following functional equation $f(az+b)=af(z)+b$.
	%\end{proof}
	
	In the following result, we examine a family of commuting transcendental entire functions. Furthermore, we show that if one of the commuting entire function is postsingularly bounded, then so is the other and their compositions.
	\begin{theorem} \label{th 2}
		Suppose $f$ and $g$ are commuting transcendental entire functions, where
		\[
		g(z) = a f^p(z) + b, \quad a,b \in \mathbb{C}, \, p\in \mathbb{N},  \, and \,\, a \neq 0,1.
		\]
	Let $f\in \mathcal{B}$.	If $f$ is postsingularly bounded, then so is $g$ and $f \circ g$.
	\end{theorem}

	\begin{proof} 
		The postsingular set of $f^p$ is given by 
		%\[
		%P \circ f(z) = a f(z) + b=g(z).
		%\]
		%It is easy to verify that $g^n = P^n \circ f^n$.

		\[
		P(f^p) = \overline {\bigcup_{n \geq 0} f^{np} \mathrm{Sing}(f^p)^{-1}  }.
		\]
		
		As $g(z) = af^p(z) + b$, we have $g^n(z)=a^nf^{np}(z)+ c_n$ where $c_n= b \sum_{k=0}^{n-1} a^k$.\\
		Also,
		\[
		\mathrm{Sing}(g^{-1}) = \mathrm{Sing} (af^p + b)^{-1} 
		= a \, \mathrm{Sing}(f^p)^{-1} + b.
		\]
		It follows that $g\in \mathcal{B}$.
		Thus,
		\[
		P(g) = \overline {\bigcup_{n \geq 0} g^n \left( \mathrm{Sing}(g^{-1}) \right)}
		= \overline {\bigcup_{n \geq 0} g^n \left( a \, \mathrm{Sing}(f^p)^{-1} + b \right)}.
		\]
		
		Now, using Lemma \ref{l2}
		\begin{align*}
			g^n \left( a \, \mathrm{Sing}(f^p)^{-1} + b \right)
			&= (a^nf^{np}(z)+ c_n)\left( a \, \mathrm{Sing}(f^p)^{-1} + b \right) \\
			&=a^nf^{np}\left( a \, \mathrm{Sing}(f^p)^{-1} + b \right)+ c_n \\
			&= a^n\left( a f^{np} \, \mathrm{Sing}(f^p)^{-1} + b \right)+ c_n \\
			%& \dots \hspace{2cm} \dots \hspace  {2cm} \dots  \\
			%& \dots \hspace{2cm} \dots \hspace  {2cm} \dots \\
			&=a^{n+1} f^{np} \, \mathrm{Sing}(f^p)^{-1} + a^nb+ c_n.
			%&= a^n \, f^n(\mathrm{Sing}(f^{-1})) + c.
		\end{align*}
		Therefore,
		\begin{align*}
			P(g)
			&= \overline {\bigcup_{n \geq 0} g^n \left( a \, \mathrm{Sing}(f^p)^{-1} + b \right)}\\
			&= \overline {\bigcup_{n \geq 0} a^{n+1} f^{np} \, \mathrm{Sing}(f^p)^{-1} + a^nb+ c_n} \\
			&= a^{r} \overline {\bigcup_{n \geq 0} f^{np} \, \mathrm{Sing}(f^p)^{-1}} + a^nb+ c_n ,  \,  \, 0<r< p,\hspace{.2cm} \text {where a is some $p^{th}$ root of unity},\\
			&= a^{r} P(f^p) + a^nb +c_n.
			%&= \overline {\bigcup_{n \geq 0} (a^{n+1} f^n +c_n)  \left( a \, \mathrm{Sing}(f^{-1}) + b \right)}\\
			%&= \overline {\bigcup_{n \geq 0}  (a^n f^n\left( a \, \mathrm{Sing}(f^{-1}) + b \right) +c_n)  }  \\
			%&= \overline {\bigcup_{n \geq 0}  (a^{n+1} f^n\left( \mathrm{Sing}(f^{-1}) \right)+b +c_n)  }
			%&=  a^{n+1} \overline {\bigcup_{n \geq 0} \left( f^n \, \mathrm{Sing}(f^{-1})\right) } +c \\
			%&= a^n \,  P(f) +b
		\end{align*}
		As $P(f)$ is bounded, it follows that $P(g)$ is likewise bounded and therefore, $g$ is postsingularly bounded. 
		Moreover, as $f$ and $g$ are commuting, we have from Lemma \refeq{l2}, $f\circ g = a f^{p+1} +b =g\circ f$ 
		%$f^l\circ g^m = f^l\big(a^m f^m + b \sum_{x=0}^{m-1} a^x\big)= a^m f^{m+l} + b \sum_{x=0}^{m-1} a^x = g^{m+l} $
		which is again postsingularly bounded, and hence the result.
	\end{proof}
%	\begin{remark}
%		%It is easy to verify that if $f$ belongs to the class $\mathcal{B}$ and $g(z)=af^p(z)+b$ with $a,b \in \mathbb{C},\, p\in \mathbb{N}$ and $|a|\neq 0$, where $f$ and $g$ are commuting, then $g$ also belongs to the class $\mathcal{B}$.\\
%		It can be easily verified that if $f$ and $g$ are commuting, $f$ belongs to class $\mathcal{B}$ and $g=af^p(z)+b$ with $a,b \in \mathbb{C}, \, p\in \mathbb{N}$ and $a \neq 0,1$, then $g$ also belongs to class $\mathcal{B}.$
%	\end{remark}
	The following example gives an illustration of Theorem \refeq{th 2}.
	\begin{example}
		Suppose $f=\lambda \sin z, \, \lambda\in\mathbb{R}\setminus\{0\}$. As $\mathrm{Sing}(f^{-1})= \{\pm \lambda\}$, so $f\in \mathcal{B}$. Then  $g=-f$ commutes with $f$. Also, we know that $\sin z$ is bounded on the real line. This together with $\lambda\in\mathbb{R}\setminus\{0\}$ implies that $f$ is postsingularly bounded. We now show that $g$ and $f\circ g$ are also postsingularly bounded.
		%As $\sin z$ is bounded on $A\subset\mathbb{C}$ for real $\lambda$, this implies $f$ is postsingularly bounded and so $g$ and $f\circ g$ must also be bounded.
		%  then for any $M>0$, we have $ P(f) = \overline {\bigcup_{n \geq 0} f^{n} \mathrm{Sing}(f)^{-1}  } \leq M$.
		We have $\mathrm{Sing}(g^{-1})= -\mathrm{Sing}(f^{-1})$. Now,
		\begin{align*}
			P(g)
			&= \overline {\bigcup_{n \geq 0} g^n \left( \mathrm{Sing}(g^{-1}) \right)}\\
			&= \overline {\bigcup_{n \geq 0} g^n \left( \mathrm{-Sing}(f^{-1}) \right)}\\
			&= \overline {\bigcup_{n \geq 0} (-1)^{n+1} f^{n} \left( \mathrm{Sing}(f^{-1}) \right)}\\
			%  	&= \overline {\bigcup_{n \geq 0} (-1)^{n+1} f^{n} \left( \mathrm{Sing}(f^{-1}) \right)}\\
			%  	&= \overline {\bigcup_{n \geq 0} (-1)^{n+1} f^{n} \left( \mathrm{Sing}(f^{-1}) \right)}\\
			%  	&= \overline {\bigcup_{n \geq 0} (-1)^{n+1} f^{n} \left( \mathrm{Sing}(f^{-1}) \right)}\\
			&=	-P(f).
		\end{align*}
		Also, $f\circ g = -f^{2}$ which implies $P(f \circ g)= -P(f)$. As $f$ is postsingularly bounded therefore, $f \circ g$ is also postsingularly bounded.
		%  \begin{align*}
			%  	P(f \circ g)
			%  	&= \overline{\bigcup_{n \geq 0}(f \circ g)^n \,\mathrm{Sing} (f \circ g)^{-1}} \\
			%  	&= \overline{\bigcup_{n \geq 0} (-1)^n f^{2n} ((\mathrm{Sing}(f)^{-1})\cup (f(\mathrm{Sing}(g)^{-1})))}\\
			%  	&= \overline{\bigcup_{n \geq 0} (-1)^n( f^{2n}\mathrm{Sing}(f)^{-1}\cup f^{2n+1} \mathrm{Sing}(g)^{-1}) }\\
			%  	&= \overline{\bigcup_{k \geq 0} (-1)^n f^k \mathrm{Sing}(f)^{-1}}\, \bigcup \, \overline{\bigcup_{l \geq 0}(-1)^{n+1} f^l \mathrm{Sing}(f)^{-1}}\\
			%  	&= P(f) \cup P(f)\\
			%  	&= P(f) \cup (-P(g)).
			%  \end{align*}
		% As $f$ and $g$ are postsingularly bounded, so is $f \circ g$ and hence, the result.
	\end{example}

		Recall that an entire function $f$ is called hyperbolic if its postsingular set $P(f)$ is a compact subset of the Fatou set $F(f)$. This requirement guarantees that the singular values of $f$ display stable dynamical behavior.
%	We now extend this concept to semigroups of transcendental entire function.
%	
%	\begin{definition}
%		A transcendental semigroup $H$ is said to be \emph{hyperbolic} if every function $h \in H$ is hyperbolic.
%	\end{definition}
%	
%	By \cite[Theorem~1.1]{rempequestion}, if $f \in \mathcal{B}$ is hyperbolic, then every connected components of the escaping set $I(f)$ are unbounded. Our goal is to generalize this statement to a suitable class of hyperbolic semigroups.
		\begin{remark}
		For two commuting transcendental entire functions $f$ and $g$, $f\in \mathcal{B}$, where $g$ has the form $af^p+b, a,b \in \C,¹ p\in \N$ and $a\neq 0,1$, arguing on similar lines, we conclude that if $f$ is hyperbolic, then so are $g$ and $f \circ g$.
		\end{remark}
	\section{Topological investigation  of abelian transcendental  semigroup and Eremenko's conjecture}\label{sec5}
		In this section, we investigate the postsingular set of an abelian transcendental semigroup. Also, we have obtained a partial result towards Eremenko's conjecture.
	As shown in \cite{schleicher2003escaping} if $f$ is an exponential map of the form
	\[
	f(z) = \lambda e^{z}, \quad \lambda \in \mathbb{C} \setminus \{0\},
	\]

then every connected component of the escaping set $I(f)$ is unbounded. That is, Eremenko’s conjecture \cite{eremenko1989iteration} is valid for exponential maps. Using the notion of postsingular set of a transcendental entire function, 
% we now extend this concept to transcendental semigroups.
	we now extend the notion of postsingular finiteness to the setting of semigroups.  
\begin{definition} A transcendental semigroup $H$ is said to be postsingularly finite if every function $h\in H$ is postsingularly finite.  
\end{definition}
From \cite[ Theorem 1.1]{rempequestion}, it follows that if $f\in \mathcal{B}$ is postsingularly finite, then every component of $I(f)$ is unbounded. In what follows, we demonstrate that this statement can be generalized to a certain class of postsingularly finite semigroups.
We now generalize Theorem \refeq{th3} to a transcendental semigroup and show that Eremenko's conjecture also holds in this case.
\begin{theorem}
	%Suppose $f \in \mathcal{B}$ and be postsingularly finite. Let $g = af^p + b,\quad a,b \in \mathbb{C},\,  p \in \mathbb{N} \,  \, and \,\, a \neq \{0,1\}$ commutes with $f$. Let $H = [f,g]$ be abelian transcendental semigroup then $H$ is also postsingularly finite and all components of $I(H)$ are unbounded.\\
	Suppose $f \in \mathcal{B}$ is postsingularly finite. Let $g = af^p + b,\, a,b \in \mathbb{C},\, p \in \mathbb{N} \, \, and \,\, a \neq 0,1$ commutes with $f$. If $H = [f,g]$, then $H$ is also postsingularly finite and all components of $I(H)$ are unbounded. 
\end{theorem}
\begin{proof}
	By using Theorem \refeq{th3}, $g$ and $f\circ g$ are both postsingularly finite since $f$ is so. Any $h\in H$ can be expressed using Theorem \refeq{Th 1} as \[h = g^n \circ f^m(z) = a^n f^{np+m}(z) + c_n\] for some $n, p, m \in \mathbb{N}$ and $c_n= b \sum_{k=0}^{n-1} a^k$. For any $h\in H$, we obtain $I(h) = I(f^{np+m}) = I(f)$. By virtue of Lemma~\refeq{l 1}, the function $h$ is postsingularly finite. As $h$ is any arbitrary element of the semigroup $H$, we have $H$ is also postsingularly finite. Every component of $I(h)$ is unbounded using \cite[Theorem 1.1]{rempequestion}. Consequently, every component of $I(H)$ is unbounded. 
\end{proof}
\begin{definition}
	A transcendental semigroup $H$ is called \emph{postsingularly
		bounded} if each element $h \in H$ is postsingularly bounded.
\end{definition}

It was proved in \cite{rempequestion} that if $f$ is an entire function of bounded type whose singular values have bounded forward orbits (that is, if $f$ is postsingularly bounded), then every connected component of the escaping set $I(f)$ is unbounded. This yields a partial answer to a conjecture of Eremenko \cite{eremenko1989iteration}. We now show that a corresponding statement also holds for a certain class of postsingularly bounded transcendental semigroups. The following result gives a generalization of Theorem \refeq{th 2} for a particular class of transcendental semigroups.
\begin{theorem}
		Suppose $f \in \mathcal{B}$ is postsingularly bounded, and $g = af^p + b \,,a,b\in \mathbb{C}\,, \, p \in \mathbb{N} \, \, and \, \, a \neq 0,1$ be permutable with $f$. Let $H = [f,g]$, then $H$ is postsingularly bounded as well, and every component of $I(H)$ is unbounded.
	\end{theorem}
	\begin{proof}
		As $f$ is postsingularly bounded, by Theorem \refeq{th 2}, $g$ and $f\circ g$ are also postsingularly bounded. Using Theorem \refeq{Th 1}, any $h\in H$ can be written as
		\[
		h = g^n \circ f^m(z) = a^n f^{np+m}(z) + c_n
		\]
		for some $n, p, m \in \mathbb{N}$ and $c_n= b \sum_{k=0}^{n-1} a^k$. As $|a|=1$ and $a\neq 1$, the term $c_n$ is bounded. We have $I(h) = I(f^{np+m}) = I(f)$ for all $h\in H$. Using Lemma \refeq{l 1}, $h$ is postsingularly bounded. As $h$ is any arbitrary element of the semigroup $H$ therefore, $H$ is also postsingularly bounded. Using \cite[Theorem 1.1]{rempequestion}, every component of $I(h)$ is unbounded, and consequently, every component of $I(H)$ is also unbounded.
	\end{proof}

		We now give a notion of hyperbolic semigroup.
	
	\begin{definition}
		A transcendental semigroup $H$ is said to be \emph{hyperbolic} if every function $h \in H$ is hyperbolic.
	\end{definition}
	
	By \cite[Theorem~1.1]{rempequestion}, if $f \in \mathcal{B}$ is hyperbolic, then every connected component of the escaping set $I(f)$ is unbounded. This can be generalized to a suitable class of hyperbolic semigroup as shown in the remark given below.
\begin{remark}
	Suppose $f \in \mathcal{B}$ is hyperbolic. Let $g = af^p + b,\, a,b \in \mathbb{C},\, p \in \mathbb{N} \, \, and \,\, a \neq 0,1$ commutes with $f$. If $H = [f,g]$, then $H$ is also hyperbolic and all components of $I(H)$ are unbounded.
\end{remark}
\end{document}